\documentclass[letterpaper,11pt]{article}
\usepackage{amsmath,amssymb,amsthm}
\usepackage{graphicx}
\usepackage{hyperref}

\newtheorem{theorem}{Theorem}

\newtheorem{lemma}{Lemma}
\newtheorem{corollary}{Corollary}

\begin{document}
\title{Hypercube orientations with only two in-degrees}
\author{Joe Buhler\thanks{Center for Communications Research, La Jolla, CA 92121 ({\tt buhler@ccrwest.org})} \and Steve Butler\thanks{Department of Mathematics, UCLA,
Los Angeles, CA 90095 ({\tt butler@math.ucla.edu}).\newline  This work was done with support of an NSF Mathematical Sciences Postdoctoral Fellowship.} \and Ron Graham\thanks{Department of Computer Science and Engineering, University of California, San Diego,
La Jolla, CA 92093 ({\tt graham@ucsd.edu}).} \and Eric Tressler\thanks{Department of Mathematics, University of California, San Diego,
La Jolla, CA 92093 ({\tt etressle@math.ucsd.edu}).}}
\date{\empty}
\maketitle

\begin{abstract}
We consider the problem of orienting the edges of the $n$-dimensional hypercube so only two different in-degrees $a$ and $b$ occur.  We show that this can be done, for two specified in-degrees, if and only if an obvious necessary condition holds.  Namely, there exist non-negative integers $s$ and $t$ so that $s+t=2^n$ and $as+bt=n2^{n-1}$.  This is connected to a question arising from constructing a strategy for a ``hat puzzle.''
\end{abstract}

\bigskip


\bigskip


\bigskip

\section{Introduction}\label{sec:introduction}
For our purposes, the $n$-dimensional hypercube, or $n$-cube, is a graph whose vertices are binary $n$-tuples, with edges joining two vertices that differ in exactly one coordinate, i.e., in the language of error-correcting codes, they have Hamming distance one.  An orientation of the $n$-cube specifies a head and tail for each edge, and the in-degree of a vertex is the number of incoming edges at that vertex.  If an orientation has only two different in-degrees, say $s$ vertices of in-degree~$a$ and $t$ vertices of in-degree~$b$, then there are
\[
s+t = 2^n
\]
vertices and
\[
as+bt = n 2^{n-1}
\]
edges (by counting edge heads).  Our main result, proved by modifying suitable Hamming codes, is that these obvious necessary conditions for the existence of an orientation with in-degrees $a$ and~$b$ --- non-negative integers $s$ and~$t$ satisfying these equations --- is in fact sufficient.

% It's important to avoid the impression that H. Iwasawa is the famous number theorist K. Iwasawa! I don't know whether they are related.  JPB
This question was implicitly raised by H.~Iwasawa \cite{iwasawa} in a different context --- finding strategies for specific hat guessing games that he introduced.  There are many puzzles involving hats or ``hat guessing games''.  Most of them have the following set-up:  A team of $n$ players has an initial strategy session where they decide on their joint strategy.  Later, a referee (or adversary) places hats colored either 0 or 1 on the players' heads, and the players can see all hat colors except that of their own hat.  Shortly thereafter, the players must simultaneously guess their hat color, with no communication between the players, or knowledge of other guesses, allowed.  A game is specified by the goal that the players try to achieve; if they are successful they each win one million dollars.  Examples include: (1) the hat placement is uniformly random and the players want to maximize the probability that all guesses are correct, (2) the players want to guarantee that at least $\lfloor n/2 \rfloor$ guesses are correct, or (3) the players want to guarantee that either all players are right or all are wrong.  Iwasawa generalizes variant (3) by asking that either exactly $a$ players are correct, or $b$ players are correct.

At the initial strategy session the first thing that the players might do is to number themselves from 1 to~$n$.  A placement of the $n$ hats can then be identified with a binary $n$-tuple, which we will think of as a vertex of the $n$-cube. When a player sees all hat colors but her own, she knows that the placement is one of the two (adjacent) vertices of the $n$-cube. A {\em strategy} is a rule that tells each player which vertex to choose.  Thus a (deterministic) strategy, to be agreed on at the team's initial strategy session is an {\em orientation} of the $n$-cube.  In other words, players agree that if after the hats are placed, they are ``on'' an edge, then they will guess as if the actual hat placement corresponds to the vertex pointed to by the arrow on that edge in the agreed-upon strategy.  For example the edge $10010\rightarrow11010$ corresponds to the situation when the second player sees $1$, $0$, $1$ and $0$ on the first, third, fourth and fifth players respectively.  The orientation of the edge indicates that the second player will predict that her hat color is~1.

If the hat placement corresponds to a vertex~$v$, then the number of correct guesses is the in-degree of~$v$.  So Iwasawa's problem of finding strategies that guarantee that {\bf either} $a$ answers are correct {\bf or} $b$ answers are correct is equivalent to the problem of finding an orientation of the hypercube where each vertex has in-degree either equal to $a$ or to~$b$.

For example, consider the special case mentioned above: $a = 0$ and $b = n$, i.e., where the goal is that everyone guesses right or everyone guesses wrong.  In this case there is an easy ``checkerboard'' winning strategy based on the fact that the hypercube is a bipartite graph.   If $v$ is a bit string $v_1v_2 \ldots v_n$,  $v_i \in \{0,1\}$, then let $P(v) \in \{0,1\}$ be the parity of the sum of the bits, i.e., $P(v) := v_1 + \cdots + v_n \bmod 2$.  All edges connect vertices of opposite parity.  So one winning strategy is to orient every edge towards, say, its endpoint with even parity.  The reader might enjoy showing that this strategy is essentially unique.

\bigskip

In the sequel $a, b, n$, and $P(v)$ will be as above.  We let $[a,b]_n$ be shorthand for the problem of realizing an orientation on the $n$-cube whose only in-degrees are $a$ and $b$.  If there are $s$ vertices with in-degree~$a$ and $t$ vertices with in-degree~$b$ then the equations above, $s+t = 2^n$, and $as+bt = n2^{n-1}$, must obviously hold.  Our main theorem says that these necessary conditions for the existence of the desired orientation are also sufficient.

\begin{theorem}\label{thm:main}
If $n$ is a positive integer, and $a$ and $b$ are between 0 and~$n$, then
there is an orientation realizing $[a,b]_n$ if and only if there are non-negative integers $s$ and $t$ so that $s+t=2^n$ and $as+bt=n2^{n-1}$.
\end{theorem}

The case in which $n$ is even and $a$ or $b$ is $n/2$ is easy to work out (every vertex will have in-degree $n/2$, and an orientation can be obtained by orienting an Eulerian tour on the $n$-cube), and it is easy to verify that this means that $a = b$ or that one of $s,t$ is zero.  We regard this case as settled and from now on take $s$ and~$t$ to be positive, and $a < b$.

\section{Reduction to primitive orientations}\label{sec:reduce}
Solving the equations $s+t=2^n$ and $as+bt=n2^{n-1}$ for $s$ and $t$ gives
\begin{equation}\label{eq:st}
s=\frac{2^{n-1}(2b-n)}{b-a}\qquad\mbox{and}\qquad t=\frac{2^{n-1}(n-2a)}{b-a}.
\end{equation}
From this we see that $a < n/2 < b$ (also obvious by interpreting the equations as implying that the average in-degree has to be $n/2$).

The first few possible cases are tabulated below.
\[
\begin{array}{|l|l|}\hline
n=1&[0,1]_1\\ \hline
n=2&[0,2]_2\\ \hline
n=3&[0,~2]_3 ,~ [0,~3]_3 ,~  [1,~2]_3 ,~  [1,~3]_3 \\ \hline
n=4&[0,~4]_4 ,~  [1,~3]_4 \\ \hline
n=5&[0,~4]_5 ,~  [0,~5]_5 ,~  [1,~3]_5 ,~  [1,~4]_5 ,~  [1,~5]_5 ,~  [2,~3]_5 ,~  [2,~4]_5\\ \hline
n=6&[0,~4]_6,~  [0,~6],~  [1,~5]_6,~  [2,~4]_6,~  [2,~6]_6 \\ \hline
n=7&[0,4]_7,~   [0,7]_7,~   [1,5]_7,~   [1,6]_7,~   [2,4]_7,~   [2,5]_7,~   [2,6]_7,~   [3,4]_7,~   [3,5]_7,~   [3,7]_7\\ \hline
n=8&[0,8]_8,~   [1,5]_8,~   [1,7]_8,~   [2,6]_8,~   [3,5]_8,~   [3,7]_8\\ \hline
\end{array}
\]
The number of possible pairs of in-degrees for the $n$-cube grows large with $n$, for instance when $n=1000$ there are $3038$ possible pairs of in-degrees that satisfy the necessary conditions.  

The following theorem will allow us to reduce the number of cases that need to be considered explicitly by giving useful reductions.

\begin{theorem}\label{thm:bootstrap}
The following hold:
\begin{itemize}
\item[(a)] $[a,b]_n$ is realizable if and only if $[n-b,n-a]_n$ is realizable;
\item[(b)] $[a,n-a]_n$ is realizable for all $0\leq a\leq n$;
\item[(c)] if $[a,b]_n$ is realizable then $[a+1,b+1]_{n+2}$ is realizable; and
\item[(d)] if $[a,b]_n$ is realizable then $[ka,kb]_{kn}$ is realizable.
\end{itemize}
\end{theorem}
\begin{proof}
For part (a) we simply note that if we have an orientation for $[a,b]_n$ then reversing all the edges gives an orientation for $[n-b,n-a]_n$ and vice versa.

For part (b) we produce an orientation by generalizing the $[0,n]_n$ strategy given above.  An edge joins vertices $v$ and $v'$ with $P(v) = 0$ and $P(v') = 1$, and $v'$ is obtained from $v$ by flipping the $i$-th coordinate for some~$i$.  We orient the edge towards $v$ if $1 \le i \le a$ and towards $v'$ if $a < i \le n$.  One checks that the in-degree of a vertex with $P(v) = 0$ is $a$, and is otherwise $n-a$.

For part (c) we note that the $(n+2)$-cube is the product of the $n$-cube with the four cycle (or $2$-cube).  Orient each copy of the $n$-cube as dictated by the orientation $[a,b]_n$.  The remain edges are a disjoint union of four-cycles and we orient each of those in (some) cyclic order.  At each vertex we have either $a$ or $b$ in-edges coming from the $n$-cube and $1$ in-edge from the four cycle, i.e., we have a solution to $[a+1,b+1]_{n+2}$.

Finally for part (d) suppose that we have an orientation for $[a,b]_n$.  Write vertices of the $kn$-cube in the form
\[
v=w_{1}w_{2}\ldots w_{n}
\]
where each $w_i$ is a vertex of the $k$-cube.  Define a map $\Psi$ from vertices of the $kn$-cube to vertices of the $n$-cube by $\Psi(v) = P(w_1)P(w_2)\ldots P(w_n)$. Note that if $u$ and $w$ are adjacent in the $kn$-cube then $\Psi(u)$ and $\Psi(w)$ are adjacent in the $n$-cube (in both cases they will differ in exactly one entry). Orient the edge between $u$ and $w$ according to the orientation between $\Psi(u)$ and $\Psi(w)$.  It is easy to check that each in-edge at a vertex in the $n$-cube gets lifted to exactly $k$ in-edges to a corresponding vertex in the $kn$-cube.  So if the in-degrees were originally $a$ and $b$ they become $ka$ and $kb$, as desired.
\end{proof}

From Theorem~\ref{thm:bootstrap} we see once we have an orientation $[a,b]_n$ that we immediately get many orientations, i.e., $[a+1,b+1]_{n+2}$, $[a+2,b+2]_{n+4}$, $[2a,2b]_{2n}$, $[3a,3b]_{3n}$, $[3a+1,3b+1]_{3n+2}$, and so on.  With a little work it is easy to tabulate the first few cases that do not seem to be reducible any further; the first few are 
\[
[0,1]_1,~[1,3]_3,~[1,5]_5,~[3,7]_7,~[1,9]_9,~[3,11]_{11},~[5,13]_{13},~[7,15]_{15}.
\]

This appears to indicate that for each odd $n$ there is a unique $a<n/2$ such that $[a,n]_n$ is not reducible using any of the results in Theorem~\ref{thm:bootstrap}, and that $a$ can be obtained from $n$ by removing the most significant bit in its binary expansion.  This is the content of the following Corollary

\begin{corollary}\label{thm:prime}
It suffices to prove Theorem~\ref{thm:main} for $[a,n]_n$ for odd $n$ where $2^k<n<2^{k+1}$ and $a = n-2^k$.
\end{corollary}

We will say that orientations as in the corollary are {\em primitive}, i.e., orientations $[a,n]_n$ where $n$ is odd, $a < 2^k < n$, and $a+2^k = n$.  Vertices with in-degree equal to~$n$ will be said to be {\em sinks}.

\begin{proof}[Proof or Corollary~\ref{thm:prime}]
We only need to show how any $[a,b]_n$ can be derived from a primitive orientation.  

First we note that the result already holds for $a+b=n$ by part (b) of Theorem~\ref{thm:bootstrap}, and so without loss of generality we may assume that $a+b>n$ (if not reverse the orientation by part (a) of the Theorem).

If $b<n$ then by part (c) of Theorem~\ref{thm:bootstrap} we have $[a,b]_n$ can be found using $[a-1,b-1]_{n-2}$.  Repeating this we see that it suffices to produce orientations for cubes when $b=n$.  Now by part (d) of Theorem~\ref{thm:bootstrap} if $[a,n]_n$ has $\gcd(a,n)>1$ then we can divide out by the $\gcd$.  Therefore, it suffices to consider orientations $[a,n]_n$ with $\gcd(a,n)=1$.  (By using the equations \eqref{eq:st} we note that in both methods of reduction that if $s$ and $t$ were positive integers for $[a,b]_n$ then they also are positive integers for $[a',b']_{n'}$.)

Let $q$ denote the odd part of $n-a$.  Then examining the equations \eqref{eq:st} we must have that $q \,\big|\, n$ and $q\,\big|\,(n-2a)$ but this implies that $q\,\big|\,\gcd (a,n)=1$.  Therefore we can conclude that $n-a=2^k$ or $a=n-2^k$ for some $k$.  This implies that $a$ and $n$ are odd (by their coprimality), and the fact that $s$ and $t$ are positive imply that $ a < n/2$ so that the orientation is primitive, as desired.
\end{proof}

\section{Using thickened Hamming balls}\label{sec:hamming}
To find primitive orientations, we will use perfect single error-correcting Hamming codes (see, e.g., \cite{hamming}).  They exist when the dimension is one less than a power of two: if $n = 2^k-1$ then there is a subset $H$ of the vertices of the $n$-cube such that every vertex is either in $H$, or adjacent to a unique element of~$H$.  In the terminology of error-correcting codes, the Hamming balls of radius~$1$ centered at elements of~$H$ are disjoint, and cover the $n$-cube.

An orientation for $[a,n]_n$ when $n = 2^k-1$, $a = n-2^{k-1} = (n-1)/2$ is easy to describe using such an~$H$.  Indeed, we let all elements of a Hamming code $H$ be sinks.  Each vertex not in $H$ is incident to exactly one edge in~$H$, and it must be oriented towards the sink.  If we erase the sinks and the edges incident to them, we have a graph where every vertex has degree $n-1 = 2^k-2$.  Since every vertex has even degree we can find an Eulerian tour (or union thereof).  If we orient those edges as we traverse those tours we get a directed graph where each vertex has in-degree (and out-degree) equal to $(n-1)/2 = a$ which, when combined with the vertices in~$H$ and edges incident on those vertices, gives the desired orientation.

To construct primitive orientations more generally, it is convenient to introduce some notation.  If $v$ is a bit vector then let $v\langle i \rangle$ denote the vector obtained from $v$ by flipping bit~$i$, let $v\langle i,j \rangle$ be shorthand for $(v\langle i \rangle)\langle j \rangle$, and similarly for more flipped bits.

If $H$ is a Hamming code in an $n$-cube, $n = 2^k-1$, then every vertex is either an element of~$H$, or is of the form $h\langle i \rangle$ for a unique $h$ in~$H$, and a unique coordinate~$i$, $1 \le i \le n$.  The following result will play a key role in the proof of Theorem~\ref{thm:main}.

\begin{lemma}
Let $H$ be a Hamming code in the $n$-cube, $n = 2^k-1$.  Fix $i$, $1 \le i \le n$, and $h$ in $H$.  Then the function 
\[ 
f(j) = \text{the unique } k \text{ such that } h\langle i,j \rangle = h' \langle k \rangle \text{ for some } h' \in H
\]
is a permutation of $\{1, \ldots, n \}$, with $i$ as a fixed point.
\end{lemma}
\begin{proof}
From the definition, $f(i) = i$.
If $j \ne i$ then $h \langle i,j \rangle$ is outside the Hamming ball of radius 1 centered at $h$, and is not in~$H$ (since the point $h \langle i \rangle$ lies in the Hamming balls centered at $h$ and $h\langle i,j \rangle$).  Therefore there are unique $h' \in H$ and $k$ such that $h' \ne h$ and $h\langle i,j \rangle = h'\langle k \rangle$.  Moreover, $k \ne i$ since this would imply that $h\langle j \rangle  = h'$.

Suppose that $h\langle i,j_1 \rangle = h_1\langle k \rangle$ and $h\langle i,j_2 \rangle = h_2\langle k \rangle$.  Then
\[
h_1\langle j_2 \rangle = h \langle i, j_1, j_2, k \rangle = h_2 \langle j_1 \rangle.
\]
By the defining properties of Hamming codes, this is possible only if $j_1 = j_2$ and $h_1 = h_2$, finishing the proof of the Lemma.
\end{proof}

Now we construct a primitive orientation $[a,n]_n$.  The key idea is to ``thicken'' Hamming balls coming from smaller dimensions.  Choose $k$ such that $2^{k}<n<2^{k+1}$, so that $n = a+2^{k}$.  Let $a = 2m-1$, and $n_0 = 2^k-1$.  Choose a Hamming code $H$ on the $n_0$-cube.
Write vertices of the $n$-cube in the form
\[
(p,v)
\]
where $p$ is on the $2m$-cube and $v$ is on the $n_0$ cube; note that $2m+n_0 = a+2^k = n$.  To describe the desired orientation we proceed as above by specifying the sinks, noting which edges have orientations forced by the location of the sinks, and then finding the remaining edge orientations by noting the existence of suitable Euler tours (see Figure~\ref{fig:steps}).

\newcommand{\HE}{$H^+$}
\newcommand{\lowE}{$\text{Low}^+$}
\newcommand{\highE}{$\text{High}^+$}
\newcommand{\HO}{$H^-$}
\newcommand{\lowO}{$\text{Low}^-$}
\newcommand{\highO}{$\text{High}^-$}

Vertices $(p,v)$ on the $n$-cube are either even, which is equivalent to $P(p) = P(v)$, or odd, which is equivalent to $P(p) \ne P(v)$.  In addition, we say that a vertex $(p,v)$ is an $H$-vertex if $v$ is in~$H$, a {\em low} vertex if $v = h \langle i \rangle$ for $h \in H$ and $1 \le i \le a$, and a {\em high} vertex if $v = h\langle i \rangle$ for $a < i \le n_0$.  Thus there are three kinds of even vertices, which we will denote \HE, \lowE, and \highE, and three kinds of odd vertices, written \HO, \lowO, and \highO.

\begin{figure}[hftb]
\centering
\includegraphics{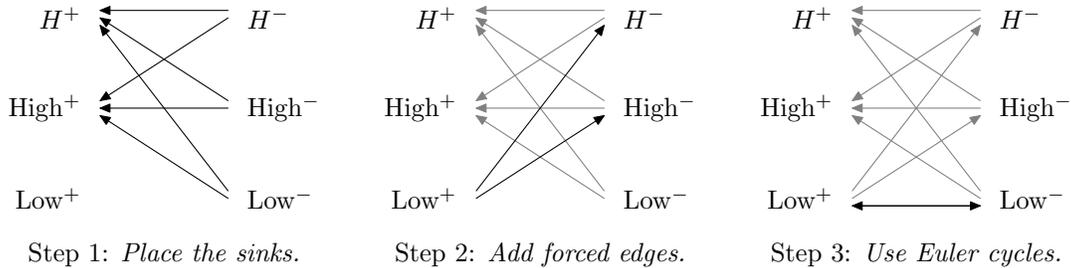}
\caption{The steps used to create the orientation for primitive $[a,n]_n$.}
\label{fig:steps}
\end{figure}

The sinks of our orientation will be the union of the \HE and \highE\ vertices.  (Note that there are no edges between these vertices.)  All edges incident on these vertices are of course oriented towards them, and they have in-degree~$n$.

Let $(p,h)$ be an \HO\ vertex.  Then $(p,h)$ is adjacent to $2m$ of the \HE\ vertices (flip a bit of~$p$).  In addition, $(p,h)$ is adjacent to  $a$ vertices of type \lowE, and $a' := n_0-a$ vertices of type \highE.  The vertices to sinks are forced, and since the in-degree of $(p,h)$ must be~$a$ we conclude that all $a$ edges going from a \lowE\ vertex to $(p,h)$ are oriented towards~$(p,h)$.

Now let $(p,h\langle i\rangle)$ be a \highO\ vertex.  Then $(p,h\langle i\rangle)$ is adjacent to a single \HO\ vertex (namely $(p,h)$) and $2m$ of the \highE\ vertices of the form $(p\langle j\rangle,h\langle i\rangle)$.  By the above lemma there are an additional $n_0-1$ neighbors of the form $(p,h'\langle \ell \rangle)$ of which $a$ of these are \lowE vertices and $n_0-1-a$ are \highE vertices.  The vertices to sinks are forced, and since the in-degree of $(p,h\langle i\rangle)$ must be~$a$ we conclude that all $a$ edges going from a \lowE\ vertex to $(p,h\langle i\rangle)$ are oriented towards~$(p,h\langle i\rangle)$.

Now consider the undirected subgraph of the $n$-cube obtained by removing all edges whose orientation has already been determined.  The only remaining edges are those which connect \lowE\ and \lowO\ vertices.  Suppose that $(p,h\langle i\rangle)$ is a \lowE\ vertex.  Then it will be adjacent to $a$ of the \lowO\ vertices of the form $(p\langle j\rangle,h\langle i\rangle)$ and by the lemma it will be adjacent to another $a$ of the \lowO\ vertices of the form $(p,h'\langle \ell\rangle)$.  A similar thing happens for the \lowO\ vertices.  In particular, in this undirected subgraph the non-isolated vertices all have degree $2a$.  Therefore by taking Euler cycle(s) we can direct these edges in the hypercube so that the in-degree of the \lowE\ and \lowO\ vertices are all $a$.  We have oriented the $n$-cube, using only in-degrees $n$ and~$a$, as desired.  This finishes the proof of Theorem~\ref{thm:main}.

To see a simple example of this constructed orientation we consider the $[1,5]_5$ case, built by thickening the Hamming balls centered at $000$ and $111$ in the 3-cube.  This is illustrated in Figure~\ref{fig:5a} (the sources are boxed and for simplicity we do not draw every possible edge). In Figure~\ref{fig:5b} we only give the in-degree $1$ vertices where we have added in all of the forced orientations, the remaining edges are easily oriented to give in-degree $1$ on all vertices.
% Should we label the vertices by type?  JPB
% We could, but it would it make messier!  SB

\begin{figure}[hftb]
\centering
\includegraphics{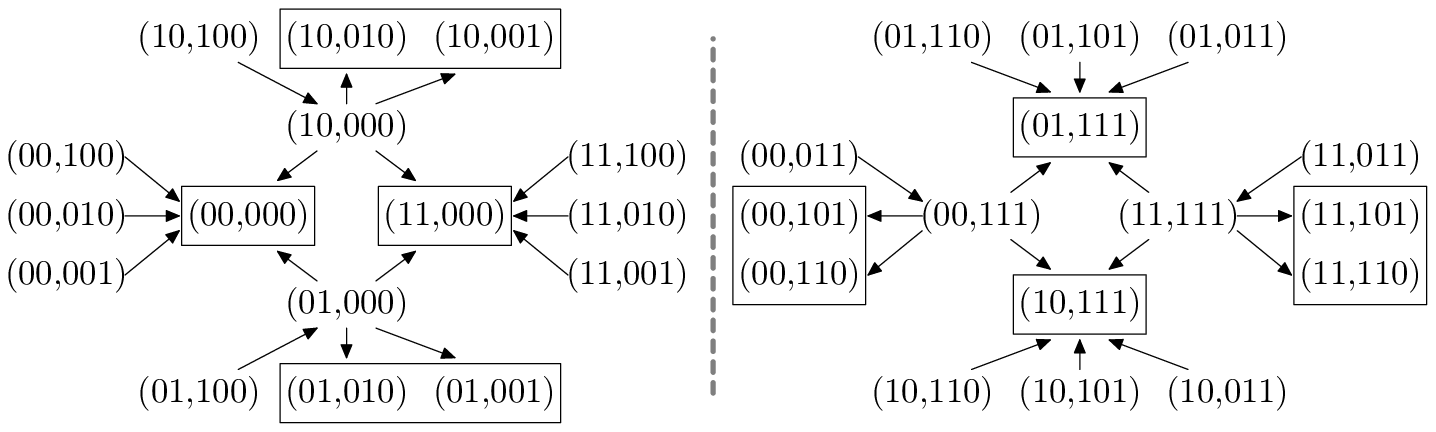}
\caption{The thickened Hamming balls for the $[1,5]_5$ case.}
\label{fig:5a}
\end{figure}

\begin{figure}[hftb]
\centering
\includegraphics{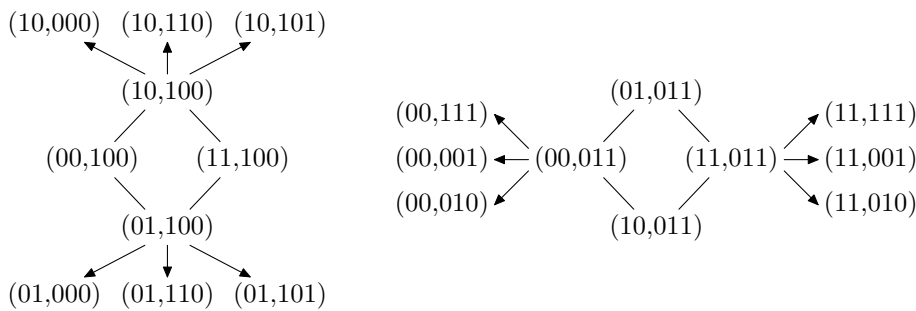}
\caption{The in-degree $1$ vertices for the $[1,5]_5$ case.}
\label{fig:5b}
\end{figure}

\section{Concluding remarks}\label{sec:conclusion}
We have looked at the problem of orienting the edges of the $n$-cube so that we have only two different in-degrees.  Our main result is that obvious necessary conditions from counting vertices and degrees are also sufficient.  In our construction we relied heavily on Hamming codes and so this argument does not extend to general graphs.  Neither does the result: consider the graph in Figure~\ref{fig:non} which has $8$ vertices and is regular of degree $3$.  Then having four vertices with in-degree $3$ and four vertices with in-degree $0$ satisfies the necessary conditions but it is easy to see that this is impossible and so it is not sufficient.  It would be interesting to know if there were any other general class of graphs for which the necessary conditions are also sufficient.

\begin{figure}[hftb]
\centering
\includegraphics[scale=2]{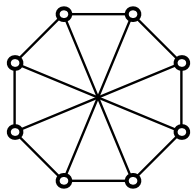}
\caption{An example of a graph where the necessary conditions are not sufficient.}
\label{fig:non}
\end{figure}

A natural generalization of the problem that we have considered is to show that we can orient the edges of the $n$-cube so that there are $s_i$ edges of in-degree $a_i$ for $i=1,\ldots,k$ if and only if $s_1+\cdots+s_k=2^n$ and $a_1s_1+\cdots+a_ks_k=n2^{n-1}$.  While our result shows this is true for $k=2$, this is false for $k=3$.  For example there is no way to orient the $4$-cube so that there are $7$ vertices with in-degree $0$, $2$ vertices with in-degree $2$, and $7$ vertices with in-degree $4$.  

One way to see this is to note that there is essentially only one way to pick seven vertices with in-degree $4$, i.e., picking vertices with Hamming weight $0$ and $2$.  This leaves only $4$ edges in the $4$-cube unoriented, namely the ones incident to the vertex with Hamming weight $4$ and there is no way to orient these $4$ edges to get $2$ vertices with in-degree $2$.

Another generalization is to increase the number of different colors of hats, i.e., to have $k$ different colors of hats.  This also has an interpretation graphically as a marking of an $n$-dimensional hyper-hypercube.  We let the vertices correspond to the $k^n$ possible placements of hats on the players,  edges again will consist of the possible $nk^{n-1}$ situations that a player can be in.  For instance if $k=3$ and $n=5$ then the edge $\{21001,21011,21021\}$, also denoted $210{*}1$ where the ${*}$ entry permutes through all possibilities, corresponds to the situation when the designated fourth player sees $2$, $1$, $0$ and $1$ on the designated first, second, third and fifth players.  To record the strategy in this situation we simply mark the vertex in the edge that corresponds to the guess the player will make.  Finally, observe that the number of correct guesses for a placement of hats corresponds to the number of times the corresponding vertex has been marked among all edges.  (The case $k=2$ reduces to what we previously considered where we mark the terminal vertex on each edge.)

In particular, a strategy that will produce either $a$ or $b$ correct guesses for any arbitrary placement of hats corresponds to marking the edges of the $n$-dimensional hyper-hypercube so that each vertex is marked either $a$ or $b$ times.  If we have $s$ vertices marked $a$ times and $t$ vertices marked $b$ times then we again have the obvious necessary condition that $s+t=k^n$ and that $as+bt=nk^{n-1}$.

Part of the difficulty of this variation is that some of our tools generalize while others do not.  For example, we can no longer ``reverse the orientation''.  However, some reductions still work.  As before let $[a,b]_n$ denote the problem of realizing a marking of the $n$-dimensional hyper-hypercube so each vertex is marked either $a$ or $b$ times.

\begin{theorem}\label{thm:bootstraph}
Given that we are working over a $k$-letter alphabet then
\begin{itemize}
\item[(a)] If $[a,b]_n$ is realizable then $[a+1,b+1]_{n+k}$ is realizable; and
\item[(b)] if $[a,b]_n$ is realizable then $[\ell a,\ell b]_{\ell n}$ is realizable.
\end{itemize}
\end{theorem}

The proofs are similar to the ones in Theorem~\ref{thm:bootstraph}. Using Theorem~\ref{thm:bootstraph} we can as before reduce to primitive orientations, the first few of which are listed below for $k=3,4,5,6,7$.

\[
\begin{array}{|c|l|}\hline
k=3&[0,1]_1,~[0,1]_2,~[0,3]_4,~[1,4]_4,~[0,3]_5,~[1,4]_5,~[0,3]_7,~[1,7]_7,~[0,3]_8 \\ \hline
k=4&[0,1]_1,~[0,1]_2,~[0,1]_3,~[0,2]_3,~[0,2]_5,~[0,4]_5,~[1,3]_5,~[1,5]_5,~[1,5]_6\\ \hline
k=5&[0,1]_1,~[0,1]_2,~[0,1]_3,~[0,1]_4,~[0,5]_6,~[1,6]_6,~[0,5]_7,~[1,6]_7,~[0,5]_8,~[1,6]_8\\ \hline
k=6&[0,1]_1,~[0,1]_2,~[0,1]_3,~[0,2]_3,~[0,1]_4,~[0,3]_4,~[0,1]_5,~[0,2]_5,~[0,3]_5,~[0,4]_5\\ \hline
k=7&[0,1]_1,~[0,1]_2,~[0,1]_3,~[0,1]_4,~[0,1]_5,~[0,1]_6,~[0,7]_8,~[1,8]_8,~[0,7]_9,~[1,8]_9\\ \hline
\end{array}
\]

A few of these are easy to check by hand (i.e., $[0,1]_1$ and $[0,1]_2$ for $k=3$).  By translating the problem of finding a marking into a {\sc sat} problem and then testing if the resulting expressions were satisfiable, we were also able to find markings for $[0,3]_4$ and $[1,4]_4$ for $k=3$ and for $[0,2]_3$ for $k=4$.  For the $[0,3]_4$ case there are $81$ vertices, $36$ are marked $3$ times each and $45$ are not marked at all.  Listed below is one way to mark the vertices with the edges to achieve $[0,3]_4$ for $k=3$, we have only listed the vertices which get marked and indicated which edges mark them by using ${*}$s, i.e., $\stackrel{*}2\stackrel{*}0\stackrel{ }2\stackrel{*}2$ indicates that $2022$ is marked by the edges ${*}022$, $2{*}22$ and $202{*}$.
\[
\begin{array}{c}
\stackrel{*}0\stackrel{*}0\stackrel{*}0\stackrel{ }0,~~
\stackrel{ }0\stackrel{*}0\stackrel{*}0\stackrel{*}2,~~
\stackrel{*}0\stackrel{*}0\stackrel{ }1\stackrel{*}0,~~
\stackrel{ }0\stackrel{*}0\stackrel{*}2\stackrel{*}1,~~
\stackrel{*}0\stackrel{*}1\stackrel{ }0\stackrel{*}1,~~
\stackrel{*}0\stackrel{*}1\stackrel{*}1\stackrel{}1,~~
\stackrel{*}0\stackrel{*}1\stackrel{ }1\stackrel{*}2,~~
\stackrel{*}0\stackrel{*}1\stackrel{*}2\stackrel{ }0,~~
\stackrel{ }0\stackrel{*}1\stackrel{*}2\stackrel{*}2,~~
\stackrel{*}0\stackrel{ }2\stackrel{*}0\stackrel{*}0,~~
\stackrel{*}0\stackrel{ }2\stackrel{*}1\stackrel{*}2,~~
\stackrel{*}0\stackrel{ }2\stackrel{*}2\stackrel{*}1,\\[3pt]
\stackrel{*}1\stackrel{*}0\stackrel{ }0\stackrel{*}1,~~
\stackrel{ }1\stackrel{*}0\stackrel{*}1\stackrel{*}2,~~
\stackrel{ }1\stackrel{*}0\stackrel{*}2\stackrel{*}0,~~
\stackrel{*}1\stackrel{*}0\stackrel{*}2\stackrel{ }1,~~
\stackrel{ }1\stackrel{*}1\stackrel{*}0\stackrel{*}2,~~
\stackrel{*}1\stackrel{*}1\stackrel{*}1\stackrel{ }0,~~
\stackrel{ }1\stackrel{*}1\stackrel{*}1\stackrel{*}1,~~
\stackrel{*}1\stackrel{*}1\stackrel{ }2\stackrel{*}2,~~
\stackrel{ }1\stackrel{*}2\stackrel{*}0\stackrel{*}0,~~
\stackrel{*}1\stackrel{ }2\stackrel{*}1\stackrel{*}1,~~
\stackrel{*}1\stackrel{ }2\stackrel{*}2\stackrel{*}2,~~
\stackrel{*}2\stackrel{*}0\stackrel{ }0\stackrel{*}2,\\[3pt]
\stackrel{*}2\stackrel{ }0\stackrel{*}1\stackrel{*}1,~~
\stackrel{*}2\stackrel{*}0\stackrel{*}1\stackrel{ }2,~~
\stackrel{*}2\stackrel{*}0\stackrel{*}2\stackrel{ }0,~~
\stackrel{*}2\stackrel{*}0\stackrel{ }2\stackrel{*}2,~~
\stackrel{*}2\stackrel{*}1\stackrel{*}0\stackrel{ }0,~~
\stackrel{*}2\stackrel{ }1\stackrel{*}0\stackrel{*}2,~~
\stackrel{ }2\stackrel{*}1\stackrel{*}1\stackrel{*}1,~~
\stackrel{*}2\stackrel{*}1\stackrel{ }2\stackrel{*}1,~~
\stackrel{*}2\stackrel{*}2\stackrel{*}0\stackrel{ }1,~~
\stackrel{*}2\stackrel{ }2\stackrel{*}0\stackrel{*}2,~~
\stackrel{*}2\stackrel{*}2\stackrel{ }1\stackrel{*}0,~~
\stackrel{*}2\stackrel{ }2\stackrel{*}2\stackrel{*}0\!.
\end{array}
\]

It may well be true that the corresponding statement of Theorem~\ref{thm:main} holds for all values of $k$ (in which case Theorem~\ref{thm:main} reduces to the $k=2$ case). For  values of $k\ge 3$ it will require some new ideas and approaches to establish the result or to search for counterexamples.

\end{document}